\documentclass[12pt]{article}
\usepackage{epsf}
\usepackage{amsbsy,amsmath}
\usepackage{mathtools}
\usepackage{mathrsfs}
\usepackage{amsfonts}
\usepackage{breqn}
\usepackage{amssymb}
\usepackage{enumitem}
\usepackage{eucal}
\usepackage{graphics,mathrsfs}
\usepackage{amsthm}
\usepackage{secdot}
\newtheorem{theorem}{Theorem}[section]
\newtheorem{lemma}[theorem]{Lemma}
\newtheorem{proposition}[theorem]{Proposition}

\newtheorem*{claim*}{Claim}
\theoremstyle{definition}
\newtheorem{definition}[theorem]{Definition}
\usepackage{enumerate}

\theoremstyle{remark}

\numberwithin{equation}{section}

\numberwithin{equation}{section}

\newsavebox{\savepar}

\begin{document}

\title{Existence of solution of the p(x)-Laplacian problem involving critical exponent and radon measure}
\author{Amita Soni and D. Choudhuri}

\date{}
\maketitle

\begin{abstract}
\noindent In this paper we are proving the existence of a nontrivial solution of the {\it{p}}(x)- Laplacian equation with Dirichlet boundary condition. We will use the variational method and concentration compactness principle involving positive radon measure $\mu$.
\begin{align*}
\begin{split}
-\Delta_{p(x)}u & = |u|^{q(x)-2}u+f(x,u)+\mu\,\,\mbox{in}\,\,\Omega,\\
u & = 0\,\, \mbox{on}\,\, \partial\Omega,
\end{split}
\end{align*}
where $\Omega \subset \mathbb{R}^N$ is a smooth bounded domain, $\mu > 0$ and $1 <  p^{-}:=\underset{x\in \Omega}{\text{inf}}\;p(x) \leq p^{+}:= \underset{x\in \Omega}{\text{sup}}\;p(x)  < q^{-}:=\underset{x\in \Omega}{\text{inf}}\;q(x)\leq q(x) \leq p^{\ast}(x) < N$. The function $f$ satisfies certain conditions. Here, $q^{\prime}(x)=\frac{q(x)}{q(x)-1}$ is the conjugate of $q(x)$ and $p^{\ast}(x)=\frac{Np(x)}{N-p(x)}$ is the Sobolev conjugate of $p(x)$.
\end{abstract}
\begin{flushleft}

\end{flushleft}

\section{Introduction}
Existence results for the problem involving the critical exponent case has been studied by many researchers, for example readers may refer \cite{Napoli}, \cite{Azorero}, \cite{Bonder1}, \cite{Zhang}, \cite{Silva}  and references therein. In \cite{Napoli} the authors have proved the existence of multiple solutions for critical case with $p$-Laplacian operator by using manifold technique. In \cite{Azorero} the authors have proved the existence and non-existence of solution of problem involving critical exponent by using concentration-compactness principle for different values of $\lambda$. In \cite{Bonder1} the authors dealt with $p(x)$-Laplacian operator with critical exponent and applied concentration-compactness principle for proving existence of solution. Many problems have also been solved with measure term in variable exponent space. For example refer \cite{Bend}, \cite{Azroul} and other references therein. In \cite{Bend}, the authors have shown the existence of a distributional solution and in \cite{Azroul} the authors have shown the existence of an entropy solution.In {\cite{Choudhuri}} we have proved the existence of multiple solution of p-Laplacian problem without Ambrosetti-Rabinowitz condition and with measure term. Motivated by this paper we are considering a similar type of problem. In this paper we are trying to extend the result by proving the existence of a nontrivial solution for {\it{p}(x)}-Laplacian problem involving an exponent $q(x)$ which is allowed to be critical in bounded domains with positive radon measure. In this work we will mainly use variational method and concentration-compactness principle. The problem which we have addressed in this article is as follows.
\begin{align*}
\begin{split}
(P):~~-\Delta_{p(x)}u & = |u|^{q(x)-2}u+f(x,u)+\mu\,\,\mbox{in}\,\,\Omega,\\
u & = 0\,\, \mbox{on}\,\, \partial\Omega,\label{main_prob}
\end{split}
\end{align*}
where $\mu > 0$ is a Radon measure and $1 < p(x) \leq  \underset{x\in \Omega}{\text{sup}\;}p(x):= p^{+} < q^{-}:=\underset{x\in \Omega}{\text{inf}}\;q(x)\leq q(x) \leq p^{\ast}(x) < N$.\\
The problem $(P)$ is new in the sense that we have tackled the presence of a Radon measure and a variable critical exponent together. The conditions assumed on the function $f$ are as follows.\\
$(f_{1})\; f(x,0)=0$ and $f$ is measurable with respect to first variable and continuous with respect to second variable.\\
$(f_{2})\; \exists$\; $c_{1}\in [p^{+},q^{-})$ s.t. $0 < c_{1}\int_{\Omega}F(x,t)dt\leq \int_{\Omega}f(x,t)t$ a.e. $x \in \Omega$ where $F(x,t) :=
\int_0^{t}f(x,s)ds$ being the primitive of $f(x,t)$.\\
$(f_{3})  \underset{\mid t\mid \rightarrow \infty}{\text{lim}}\frac{f(x,t)}{{\mid t\mid}^{q(x)-1}}= 0$ uniformly a.e. $x \in \Omega$.\\
An example of a function satisfying the above conditions is 
$f(x,t)=|t|^{r(x)-1}$;\\ where $c_{1} < r(x) < q(x)$.\\
Throughout this article, we will denote the measure of a measurable set $E$ of $\Omega$ by $|E|$ and the absolute value of any real number, say $a$, as $|a|$. We will denote ${\parallel .\parallel}_{W_{0}^{1,p(x)}}= {\parallel .\parallel}$.

\begin{theorem}
Suppose that $(f_{1})-(f_{3})$ hold. Then problem $(P)$ possesses a nontrivial weak solution.
\end{theorem}

\section{Premilinaries}
\subsection{Definitions}
\begin{definition}\label{defn1}
	Let $(\mu_n)$ be a bounded sequence of measures in
	$\mathfrak{M}(\Omega)$. We say that $(\mu_n)$ converges to a measure $\mu \in \mathfrak{M}(\Omega)$ in the sense of measure if 
	$$\displaystyle{\int_{\Omega}\phi d\mu_n} \rightarrow
	\int_{\Omega}\phi d\mu\,\, \,\,\,\forall\,\,\phi\in
	C_0(\bar{\Omega}).$$ We denote this convergence by $\mu_n
	\xrightharpoonup{} \mu$. The topology defined via this weak
	convergence is metrizable and a bounded sequence with respect to
	this topology is pre-compact.
\end{definition}
\begin{definition}
The Marcinkiewicz space  ${M}^q(\Omega)$ \cite{Benilan} (or the weak $L^q(\Omega)$ space)  defined for every $0 < q <\infty$, as the space of all measurable functions $f:\Omega\rightarrow \mathbb{R}$ such that the corresponding distribution satisfy an estimate of the form
$$|\{x\in \Omega:|f(x)|>t\}|\leq \frac{C}{t^q},\hspace{0.4cm}t>0,C<\infty.$$ 
\end{definition}
For bounded $\Omega$ we have ${M}^q\subset {M}^{\bar{q}}$ if $q\geq \bar{q}$, for some fixed positive $\bar{q}$. We recall here the following useful continuous embeddings
\begin{equation}\label{mar}
L^q(\Omega)\hookrightarrow {M}^q(\Omega)\hookrightarrow L^{q-\epsilon}(\Omega),
\end{equation}
for every $1<q<\infty$ and $0<\epsilon<q-1$.\\

\subsection{Variable exponent Sobolev space}
For each open subset $\Omega \subset \mathbb{R}^{N}(N \geq 2)$, we define $C_{+}(\overline{\Omega})=\lbrace p\;|\;p\in C(\overline{\Omega}), p(x) > 1\; {\text{for any}}\; x\in \overline{\Omega}\rbrace$ and $1 < p^{-}:=\underset{x\in \Omega}{\text{inf}}\;p(x)\leq \underset{x\in\Omega}{\text{sup}}\;p(x)=: p^{+} < N.$ The variable exponent Lebesgue space $L^{p(x)}(\Omega)$ is defined by\\
$$L^{p(x)}(\Omega)=\left\lbrace u:\Omega\rightarrow\mathbb{R}\;|\;u \;{\text{is measurable and}} \int_{\Omega}|u|^{p(x)}dx < \infty\right\rbrace$$ endowed with the norm (the Luxemburg norm)
 $|u|_{p(x)}={\text{inf}}\lbrace \lambda > 0\;|\;\int_{\Omega}|\frac{u}{\lambda}|^{p(x)}dx\leq 1\rbrace$.\\
We will define variable exponent Sobolev space as $$W^{1,p(x)}(\Omega)=\lbrace u\in L^{p(x)}(\Omega)\;|\;|\nabla u|\in L^{p(x)}(\Omega)\rbrace$$ with the norm $\|u\|_{1,p(x)}=|u|_{p(x)}+|\nabla u|_{p(x)}$.
With these norms, $L^{p(x)}(\Omega)$ and $W^{1,p(x)}(\Omega)$ are separable reflexive Banach spaces(\cite{Kovacik}). For $p(x)\equiv p,\; p(x)$- Laplacian reduces to $p$-Laplacian.
\begin{proposition} 
Set $\rho(u):=\int_{\Omega}|u(x)|^{p(x)}dx$. For $u\in L^{p(x)}(\Omega)$ and $(u_{n})_{n\in\mathbb{N}}\subset L^{p(x)}(\Omega)$, we have
\begin{itemize}
\item $u\neq 0\Rightarrow \|u\|_{L^{p(x)}}(\Omega)=\lambda$ iff $\rho(\frac{u}{\lambda})=1$,
\item $\|u\|_{L^{p(x)}(\Omega)} < 1(=1;> 1)\Leftrightarrow \rho(u) < 1(=1; >1)$,
\item $\|u\|_{L^{p(x)}(\Omega)} < 1\Rightarrow \|u\|_{L^{p(x)}(\Omega)}^{p^{+}}\leq \rho(u)\leq \|u\|_{L^{p(x)}(\Omega)}^{p^{-}},$
\item $\|u\|_{L^{p(x)}(\Omega)} > 1\Rightarrow \|u\|_{L^{p(x)}(\Omega)}^{p^{-}}\leq \rho(u)\leq \|u\|_{L^{p(x)}(\Omega)}^{p^{+}},$
\item $\underset{n\rightarrow\infty}{\text{lim}}||u_{n}||_{L^{p(x)}(\Omega)}=0(\infty)\Leftrightarrow \underset{n\rightarrow\infty}{\text{lim}}\rho(u_{n})=0(\infty)$.
\end{itemize}
\end{proposition}
We state the generalized H\"{o}lder inequality and embedding results in the following propositions (\cite{Fan}, \cite{Kovacik}, \cite{Zhao}, \cite{Diening}).
\begin{proposition}
For any $u\in L^{p(x)}(\Omega)$ and $v\in L^{p^{\prime}(x)}(\Omega)$, where $L^{p^{\prime}(x)}(\Omega)$ is the conjugate space of $L^{p(x)}(\Omega)$ such that $\frac{1}{p(x)}+\frac{1}{p^{\prime}(x)}=1$,\\
$\big|\int_{\Omega}uv\;dx\big|\leq \left(\frac{1}{p^{-}}+\frac{1}{p^{-\prime}}\right) \|u\|_{p(x)}\|v\|_{p^{\prime}(x)}$
\end{proposition}

\begin{proposition}
(i) If $q\in C_{+}(\overline{\Omega})$ and $q(x) < p^{\ast}(x)$ for any $x\in\overline{\Omega}$, then $W^{1,p(x)}(\Omega)\hookrightarrow L^{q(x)}(\Omega)$ is compact and continuous.\\
(ii) There exists a constant $c > 0$ such that $\|u\|_{p(x)} \leq c\|\nabla u\|_{p(x)}\;\forall\; u\in W_{0}^{1,p(x)}(\Omega)$.
\end{proposition}

\subsection{Functional analytic setup}
We first consider a sequence of problems $(P_{n})$ which are as follows.
\begin{align*}
\begin{split}
-\Delta_{p(x)}u & = |u|^{q(x)-2}u+f(x,u)+\mu_{n}\,\,\mbox{in}\,\,\Omega,\\
u & = 0\,\, \mbox{on}\,\, \partial\Omega,
\end{split}
\end{align*}
where $\mu_{n}$ are smooth functions such that $\mu_{n} \rightharpoonup \mu$ in measure in the sense of definition 2.1.\\
The corresponding energy functional to the sequence of problems $(P_{n})$ is given as
\begin{equation*} 
I_{n}(u)= \int_{\Omega}\frac{{|\nabla u|}^{p(x)}}{p(x)}dx -\int_{\Omega} \frac{{|u|}^{q(x)}}{q(x)}dx - \int_{\Omega}F(x,u) dx - \int_{\Omega}u d\mu_{n}.
\end{equation*}
The Fr\'{e}chet derivative of $I_{n}$ is defined as
\begin{equation*}
<I_{n}^{\prime}(u),v> = \int_{\Omega}{|\nabla u|}^{p(x)-2}\nabla u\nabla v dx - \int_{\Omega}{|u|}^{q(x)-2}uv dx - \int_{\Omega}f(x,u)v dx - \int_{\Omega}\mu_{n}v dx
\end{equation*}
$\forall u,v \in T$, where $T=W^{1, p(x)}(\Omega)\cap C_0(\bar{\Omega})$, $C_0(\bar{\Omega})=\{\varphi\in C(\bar{\Omega}):\varphi|_{\partial\Omega}=0\}$ and $C(\bar{\Omega})$ will denote the space of continuous functions over $\bar{\Omega}$. We now define the corresponding energy functional of the problem $(P)$ as 
\begin{equation*}
I(u)=\int_{\Omega}\frac{{|\nabla u|}^{p(x)}}{p(x)}dx -\int_{\Omega} \frac{{|u|}^{q(x)}}{q(x)}dx - \int_{\Omega}F(x,u) dx - \int_{\Omega}u d\mu 
\end{equation*}
and its Fr\'{e}chet derivative as 
\begin{equation*}
<I^{\prime}(u),v> = \int_{\Omega}{|\nabla u|}^{p(x)-2}\nabla u\nabla v dx -\int_{\Omega}{|u|}^{q(x)-2}uv dx - \int_{\Omega}f(x,u)v dx - \int_{\Omega}v d\mu
\end{equation*}
for every $u,v\in T^{\prime}$, where $T^{\prime}=W^{1, s(x)}(\Omega)\cap C_0(\bar{\Omega})$ and $1\leq s(x) < s={\text{min}}\left\lbrace 1-\frac{1}{\gamma},{\frac{(\gamma -1)p^{-}}{2\gamma-1}}\right\rbrace$.
\begin{definition}
$u\in W_{0}^{1,s(x)}$ is said to be a weak solution of the problem $(P)$ if
\begin{equation*}
\int_{\Omega}|\nabla u|^{p(x)-2}\nabla u\nabla\varphi dxdy - \int_\Omega |u|^{q-2}u\varphi dx - \int_\Omega f(x,u)\varphi dx - \int_{\Omega}\varphi d\mu =0,
\end{equation*}
\end{definition}
$\forall$\; $\varphi \in T^{\prime}$.\\
 \section{Existence Results}
To prove the main result of this paper which is given in form of Theorem 1.1, we need to first prove few lemmas related to the mountain pass theorem and Palais-Smale condition. It is clear that $I_{n}$ is $C^{1}$ functional\; $\forall\;n \geq 1$.

\begin{lemma}\label{lem1}
The functional $I_{n}$ satisfies mountain pass geometry in the sense that:
\begin{itemize}
\item $I_{n}(0)=0$
\item $\exists\; r, \eta > 0$ such that $I_{n}(u)\geq\eta$ if $\|u\| > r$.
\item $\exists\; u, \|u\| > r$ such that $I_{n}(u)\leq 0$.
\end{itemize}
\end{lemma}

\begin{proof}
$I_{n}(0)=0$ is obvious. For proving 2., we need the assumptions $(f_{2})$ and $(f_{3})$. From these assumptions we obtain, 
\begin{align*}
\begin{split}
c_{1}\int_{\Omega}F(x,u)dx&\leq \int_{\Omega}f(x,u)u dx\\
&\leq \int_{\Omega}|f(x,u)u|dx\\
&\leq \epsilon\int_{\Omega}|u|^{q(x)}dx+m(\epsilon)\\
&\leq(\epsilon+m(\epsilon))\int_{\Omega}|u|^{q(x)}dx
\end{split}
\end{align*}
This implies $\int_{\Omega}F(x,u)\leq \left(\frac{\epsilon+m(\epsilon)}{c_{1}}\right)\int_{\Omega}|u|^{q(x)}dx.$
Choose $\|u\|=r$ sufficiently small so that$\int_{\Omega}|u|^{q(x)}dx\leq \|u\|_{q(x)}^{q^-}$ since $\|u\|=r < 1$. Now using the Poincar\'{e} inequality, H\"{o}lder inequality and continuous embedding of $W_{0}^{1,p(x)}(\Omega)$ into $L^{q(x)}(\Omega)$, we have
\begin{align*}
\begin{split}
I_{n}(u)&=\int_{\Omega}\frac{{|\nabla u|}^{p(x)}}{p(x)}dx-\int_{\Omega}\frac{|u|^{q(x)}}{q(x)}dx-\int_{\Omega}F(x,u)dx-\int_{\Omega}\mu_{n}udx\\
&\geq \frac{c}{p^{+}}|r|^{p^{+}}-\frac{1}{q^{-}}|r|_{q(x)}^{q^{-}}-\left(\frac{\epsilon+m(\epsilon)}{c_{1}}\right)|r|_{q(x)}^{q^{-}}-\|\mu_{n}\|_{q^{\prime}(x)}\|u\|_{q(x)}\\
&\geq \frac{c}{p^{+}}|r|^{p^{+}}-\frac{1}{q^{-}}|r|^{q^{-}}-\left(\frac{\epsilon+m(\epsilon)}{c_{1}}\right)|r|^{q^{-}}-\|\mu_{n}\|_{q^{\prime}(x)}|r|^{q^{-}}\\
&=\frac{cr^{p^{+}}}{p^{+}}-r^{q^{-}}\left\lbrace\frac{1}{q^{-}}-\left(\frac{\epsilon+m(\epsilon)}{c_{1}}\right)-\|\mu_{n}\|_{q^{\prime}(x)}\right\rbrace 
\end{split}
\end{align*}
Since $q^{-} > p^{+}$ so $I_{n}(u)\geq\eta$ for some $\eta > 0$. We can prove the 3. by using the assumption $(f_{2})$, for $t > 0$ and $u\neq 0$ consider,
\begin{align*}
\begin{split}
I_{n}(tu)&= \int_{\Omega}\frac{1}{p(x)}{|\nabla tu|}^{p(x)}dx - \int_{\Omega}\frac{1}{q(x)}{|tu|}^{q(x)}dx - \int_{\Omega}F(x,tu) dx - \int_{\Omega}tu d\mu_{n}\\ 
&= \int_{\Omega}\frac{t^{p(x)}}{p(x)}{|\nabla tu|}^{p(x)}dx - \int_{\Omega}\frac{t^{q(x)}}{q(x)}{|tu|}^{q(x)}dx - \int_{\Omega}F(x,tu) dx - \int_{\Omega}tu d\mu_{n}\\
&\leq \int_{\Omega}\frac{t^{p(x)}}{p(x)}{|\nabla tu|}^{p(x)}dx - \int_{\Omega}\frac{t^{q(x)}}{q(x)}{|tu|}^{q(x)}dx - \int_{\Omega}tu d\mu_{n}
\end{split}
\end{align*}
This implies
\begin{align}
I_{n}(tu)\leq t^{p^{+}}\int_{\Omega}\frac{1}{p(x)}{|\nabla u|}^{p(x)}dx-t^{q^{-}}\int_{\Omega}\frac{1}{q(x)}{|u|}^{q(x)}dx-t\int_{\Omega}\mu_{n}u dx.
\end{align}
On dividing (3.1) by $t^{p^+}$ and passing the limit $t\rightarrow \infty$ we get, $I_{n}(tu)\rightarrow -\infty$ since $q^{-} > p^{+}$.\\
Hence, $I_{n}(u)$ satisfies the hypothesis of mountain pass theorem.
\end{proof}

\begin{lemma}\label{lem2}
The functional $I_{n}$ satisfies Palais-Smale condition.
\end{lemma}

\begin{proof}
Let $(u_{m,n})$ be a Palais-Smale sequence such that $I(u_{m,n})\rightarrow c$ and $I^{\prime}(u_{m,n})\rightarrow 0$ in $(W_{0}^{1,p(x)}(\Omega))^{\prime}.$ We first show that $(u_{m,n})$ is bounded in $W_{0}^{1,p(x)}(\Omega)$. We will prove it by contradiction. Let $\|u_{m,n}\|\rightarrow \infty$ as $m\rightarrow\infty$. Then we have,
\begin{align*}
\begin{split}
I_{n}(u_{m,n})-\frac{1}{c_{1}}\langle I^{\prime}(u_{m,n}),u_{m,n}\rangle&=\int_{\Omega}\frac{|\nabla u_{m,n}|^{p(x)}}{p(x)}dx-\frac{1}{c_{1}} \int_{\Omega}|\nabla u_{m,n}|^{p(x)}dx-\int_{\Omega}\frac{|u_{m,n}|^{q(x)}}{q(x)}dx\\
&+\frac{1}{c_{1}}\int_{\Omega}|u_{m,n}|^{q(x)}dx-\int_{\Omega}F(x,u_{m,n})dx
-\int_{\Omega}u_{m,n}\mu_{n}dx\\
&+\frac{1}{c_{1}}\int_{\Omega}f(x,u_{m,n})u_{m,n}dx+\frac{1}{c_{1}}\int_{\Omega}u_{m,n}\mu_{n}dx\\
\geq& \frac{1}{p^{+}}\int_{\Omega}{|\nabla u_{m,n}|^{p(x)}}-\frac{1}{c_{1}} \int_{\Omega}|\nabla u_{m,n}|^{p(x)}dx-\frac{1}{q^{-}}\int_{\Omega}|u_{m,n}|^{q(x)}dx\\
&+\frac{1}{c_{1}}\int_{\Omega}|u_{m,n}|^{q(x)}dx-\int_{\Omega}F(x,u_{m,n})dx
-\int_{\Omega}u_{m,n}\mu_{n}dx\\
&+\frac{1}{c_{1}}\int_{\Omega}f(x,u_{m,n})u_{m,n}dx+\frac{1}{c_{1}}\int_{\Omega}u_{m,n}\mu_{n}dx\\
=&\left(\frac{1}{p^{+}}-\frac{1}{c_{1}}\right)\int_{\Omega}{|\nabla u_{m,n}|^{p(x)}}dx+\left(\frac{1}{c_{1}}-\frac{1}{q^{-}}\right)\int_{\Omega}|u_{m,n}|^{q(x)}dx\\
&+\frac{1}{c_{1}}\left(\int_{\Omega}f(x,u_{m,n})u_{m,n}dx-c_{1}\int_{\Omega}F(x,u_{m,n})dx\right)\\
&-\left(1-\frac{1}{c_{1}}\right)\int_{\Omega}u_{m,n}\mu_{n}dx
\end{split}
\end{align*}
Using the assumption $(f_{2})$, we get
\begin{align*}
\begin{split}
I_{n}(u_{m,n})-\frac{1}{c_{1}}\langle I^{\prime}(u_{m,n}),u_{m,n}\rangle=\left(\frac{1}{p^{+}}-\frac{1}{c_{1}}\right)\int_{\Omega}{|\nabla u_{m,n}|^{p(x)}}dx-A\int_{\Omega}u_{m,n}\mu_{n}dx
\end{split}
\end{align*}
where $A=\left(1-\frac{1}{c_{1}}\right) > 0$. Furthermore on applying the Poincar\'{e} inequality, H\"{o}lder inequality, embedding of $W_{0}^{1,p(x)}(\Omega)$ into $L^{p(x)}(\Omega)$ and the fact that $\|\nabla u\|_{p(x)}$ and $\|u\|_{1,p(x)}$ are equivalent norm on $W_{0}^{1,p(x)}(\Omega)$ we get
\begin{align}
\begin{split}
I_{n}(u_{m,n})-\frac{1}{c_{1}}\langle I^{\prime}(u_{m,n}),u_{m,n}\rangle\geq
c^{\prime}\left(\frac{1}{p^{+}}-\frac{1}{c_{1}}\right)\|u_{m,n}\|^{p^{-}}-A\|\mu_{n}\|_{p^{\prime}(x)}\|u_{m,n}\|
\end{split}
\end{align}
Now on dividing both sides of (3.2) by $\|u_{m,n}\|$ and passing the limit $m\rightarrow\infty$ we get $0\geq\infty$ as $p^{-} > 1$ which is absurd. Hence $(u_{m,n})$ is bounded in $W_{0}^{1,p(x)}(\Omega)$. Since this is a reflexive space, there exists a subsequence say $(u_{m,n})$ which converges weakly to $u_{n}$ in $W_{0}^{1,p(x)}(\Omega)$. To prove this convergence to be a strong convergence we will use the concentration compactness principle for variable exponents (refer \cite{Bonder2}, Theorem 1.1) from which we have\\
$|u_{m,n}|^{q(x)}\rightharpoonup\nu=|u_{n}|^{q(x)}+\sum_{i\in J}\nu_{i}\delta x_{i}$\\
$|\nabla u_{m,n}|^{p(x)}\rightharpoonup \mu\geq |\nabla u_{n}|^{p(x)}+\sum_{i\in J}\mu_{i}\delta x_{i}$\\
$ S\nu_{i}^{\frac{1}{p^{\ast}(x_{i})}}\leq\mu_{i}^{\frac{1}{p(x_{i})}}\;\forall\; i\in J$, where $S :=\underset{\phi\in C_{0}^{\infty}(\Omega)}{\text {inf}}\frac{{\parallel |\nabla\phi| \parallel}_{L^{p(x)}(\Omega)}}{\|\phi\|_{L^{q(x)}(\Omega)}}$ and $J$ is a finite set. $(\nu_{i})_{i\in J}$ and $(\mu_{i})_{i\in J}$ are positive numbers and points $(x_{i})_{i\in J}$ belongs to the critical set\\
 $A=\lbrace x\in \Omega: q(x)=p^{\ast}(x)\rbrace$.
\begin{claim*}
$J$ is empty.
\end{claim*}
\begin{proof}
Let $J\neq \phi$. Define $\psi_{i,\epsilon}(x)=\psi\left(\frac{x-x_{i}}{\epsilon}\right)$. 
\begin{align*}
\begin{split}
0=&\langle I^{\prime}(u_{m,n}),u_{m,n}\psi_{i,\epsilon}\rangle\\
=&\int_{\Omega}{|\nabla u_{m,n}|}^{p(x)-2}\nabla u_{m,n}\nabla(u_{m,n}\psi_{i,\epsilon})dx-\int_{\Omega}|u_{m,n}|^{q(x)-2}u_{m,n} u_{m,n}\psi_{i,\epsilon}dx\\
&-\int_{\Omega}f(x,u_{m,n})u_{m,n}\psi_{i,\epsilon}dx-\int_{\Omega}\mu_{n}u_{m,n}\psi_{i,\epsilon}dx\\
=& \int_{\Omega}{|\nabla u_{m,n}|}^{p(x)}\nabla\psi_{i,\epsilon}dx+\int_{\Omega}({|\nabla u_{m,n}|}^{p(x)-2}\nabla u_{m,n}\nabla\psi_{i,\epsilon})u_{m,n}dx\\
&-\int_{\Omega}|u_{m,n}|^{q(x)}\psi_{i,\epsilon}dx-\int_{\Omega}f(x,u_{m,n})u_{m,n}\psi_{i,\epsilon}dx-\int_{\Omega}\mu_{n}u_{m,n}\psi_{i,\epsilon}dx
\end{split}
\end{align*}
In addition to this we have
\begin{align}
\begin{split}
0=\underset{m\rightarrow\infty}{\text{lim}}&\langle I^{\prime}(u_{m,n}),\phi\rangle\\
=\underset{m\rightarrow\infty}{\text{lim}}&\left[\int_{\Omega}{|\nabla u_{m,n}|}^{p(x)-2}\nabla u_{m,n}\nabla\phi\;dx-\int_{\Omega}|u_{m,n}|^{q(x)-2}u_{m,n}\phi\;dx\right.\\
&\left.-\int_{\Omega}f(x,u_{m,n})\phi\;dx-\int_{\Omega}\mu_{n}\phi\;dx\right] 
\end{split}
\end{align}
We also have 
\begin{align}
\begin{split}
|\nabla u_{m,n}|^{p(x)-2}\nabla u_{m,n}&\rightharpoonup |\nabla u_{n}|^{p(x)-2}\nabla u_{n}\; \text{in}\; L^{p^{\prime}(x)}(\Omega)\\
|u_{m,n}|^{q(x)-2}u_{m,n}&\rightharpoonup |u_{n}|^{q(x)-2}u_{n}\;\; \;\;\;\;\;\text{in}\; L^{q^{\prime}(x)}(\Omega)
\end{split}
\end{align}
Since, $u_{m,n}\rightharpoonup u_{n}$ in $W_{0}^{1,p(x)}(\Omega)$ and $W_{0}^{1,p(x)}(\Omega)$ is compactly embedded in $L^{p(x)}(\Omega)$ hence $u_{m,n}\rightarrow u_{n}$ in $L^{p(x)}(\Omega)$. By the Egoroff's theorem, $u_{m,n}\rightarrow u_{n}$ a.e. in $\Omega$ upto a subsequence. Also, by the continuity of $f$ with respect to the second variable we conclude
 $f(x,u_{m,n})\rightarrow f(x,u_{n})$ which implies $\int_{\Omega}f(x,u_{m,n})\phi\rightarrow\int_{\Omega}f(x,u_{n})\phi$ by the Dominated convergence theorem.\\
Hence from (3.3), 
\begin{align*}
\begin{split}
\int_{\Omega}{|\nabla u_{n}|}^{p(x)-2}\nabla u_{n}\nabla\phi\;dx-\int_{\Omega}|u_{n}|^{q(x)-2}u_{n}\phi\;dx-\int_{\Omega}f(x,u_{n})\phi\;dx-\int_{\Omega}\mu_{n}\phi\;dx=0
\end{split}
\end{align*}
implying that $u_{n}$ is a weak solution to the sequence of problems $(P_{n})$. Thus,
\begin{align}
\begin{split}
0=\langle I^{\prime}(u_{n}),u_{n}\psi_{i,\epsilon}\rangle=&\int_{\Omega}|\nabla u_{n}|^{p(x)}\nabla \psi_{i,\epsilon}dx+\int_{\Omega}({|\nabla u_{n}|}^{p(x)-2}\nabla u_{n}\nabla\psi_{i,\epsilon})u_{n}dx\\
&-\int_{\Omega}|u_{n}|^{q(x)}\psi_{i,\epsilon}dx-\int_{\Omega}f(x,u_{n})u_{n}\psi_{i,\epsilon}dx-\int_{\Omega}\mu_{n}u_{n}\psi_{i,\epsilon}dx
\end{split}
\end{align}
Substituting $\phi=u_{m,n}\psi_{i,\epsilon}$ in (3.3) and then subtracting (3.3) from (3.5), we get
\begin{equation}
\begin{split}
0=\underset{m\rightarrow\infty}{\text{lim}}&\langle I^{\prime}(u_{m,n}),u_{m,n}\psi_{i,\epsilon}\rangle-\langle I^{\prime}(u_{n}),u_{n}\psi_{i,\epsilon}\rangle\\
=\underset{m\rightarrow\infty}{\text{lim}}&\left[\int_{\Omega}|\nabla u_{n}|^{p(x)}\nabla \psi_{i,\epsilon}dx+\int_{\Omega}({|\nabla u_{m,n}|}^{p(x)-2}\nabla u_{m,n}\nabla\psi_{i,\epsilon})u_{m,n}\;dx\right.\\
&-\int_{\Omega}|u_{m,n}|^{q(x)}\psi_{i,\epsilon}\;dx-\int_{\Omega}f(x,u_{m,n})\psi_{i,\epsilon}\;dx\left.-\int_{\Omega}\mu_{n}u_{m,n}\psi_{i,\epsilon}\;dx\right]\\
&-\left[\int_{\Omega}|\nabla u_{n}|^{p(x)}\nabla \psi_{i,\epsilon}dx-\int_{\Omega}\mu_{n}u_{n}\psi_{i,\epsilon}dx-\int_{\Omega}|u_{n}|^{q(x)}\psi_{i,\epsilon}dx\right.\\
&\left.-\int_{\Omega}f(x,u_{n})u_{n}\psi_{i,\epsilon}dx+\int_{\Omega}({|\nabla u_{n}|}^{p(x)-2}\nabla u_{n}\nabla\psi_{i,\epsilon})u_{n}dx\right]
\end{split}
\end{equation}
From $(f_{3})$ we have, $|f(x,t)t| < \frac{\epsilon}{2\tilde{c}}t^{q(x)}+m(\epsilon)\;\forall\;t \in\mathbb{R}$ and a.e. in $\Omega$. Let $\|u\|_{q(x)}^{q^{+}}=\tilde{c}$. Choose $\delta=\frac{\epsilon}{2m(\epsilon)} > 0$ and $F\subseteq\Omega$ such that $|F| < \delta$. Then
\begin{align*}
\begin{split}
\bigg|\int_{F}f(x,u_{m,n})u_{m,n}dx\bigg| &\leq \int_{\Omega}\big|f(x,u_{m,n})u_{m,n}\big|dx\\
&\leq \int_{F}m(\epsilon)dx+\frac{\epsilon}{2\tilde{c}}\int_{F}|u|^{q(x)}dx\\
&\leq m(\epsilon)|F|+\frac{\epsilon}{2\tilde{c}}\|u\|_{q(x)}^{q^{+}}\\
& < m(\epsilon)\frac{\epsilon}{2m(\epsilon)}+\frac{\epsilon}{2\tilde{c}}\tilde{c}\\
&=\epsilon
\end{split}
\end{align*}
Hence, $\lbrace f(x,u_{m,n})u_{m,n}dx:m\in\mathbb{N}\rbrace$ is equiabsolutely continuous and therefore by the Vitali convergence theorem $\int_{\Omega}f(x,u_{m,n})u_{m,n}dx\rightarrow \int_{\Omega}f(x,u_{n})u_{n}dx$ as $ m\rightarrow\infty$. This implies $\int_{\Omega}f(x,u_{m,n})u_{m,n}\psi_{i,\epsilon}dx\rightarrow \int_{\Omega}f(x,u_{n})u_{n}\psi_{i,\epsilon}dx$ as $ m\rightarrow\infty$. We further have from (3.4) and weak convergence of $(u_{m,n})$ that
\begin{align*}
\begin{split}
\int_{\Omega}|\nabla u_{m,n}|^{p(x)-2}\nabla u_{m,n}\nabla\psi_{i,\epsilon}dx&\rightarrow \int_{\Omega}|\nabla u_{n}|^{p(x)-2}\nabla u_{n}\nabla\psi_{i,\epsilon}dx\\
\int_{\Omega}\mu_{n}\psi_{i,\epsilon}u_{m,n}dx&\rightarrow \mu_{n}\psi_{i,\epsilon}u_{n}dx.
\end{split}
\end{align*}
Using these results in (3.6),
\begin{align}
\begin{split}
0=&\underset{m\rightarrow\infty}{\text{lim}}\left(\int_{\Omega}|\nabla u_{m,n}|^{p(x)}\nabla\psi_{i,\epsilon}dx-\int_{\Omega}|u_{m,n}|^{q(x)}\psi_{i,\epsilon}dx\right)\\
&-\left( \int_{\Omega}|\nabla u_{n}|^{p(x)}\nabla\psi_{i,\epsilon}dx-\int_{\Omega}|u_{n}|^{q(x)}\psi_{i,\epsilon}dx\right)
\end{split}
\end{align}
Now on applying concentration compactness principle in (3.7), we have \\
$\psi_{i,\epsilon}d\mu-\psi_{i,\epsilon}d\nu=0$. As $\epsilon\rightarrow 0, \mu_{i}=\nu_{i}$.\\
Again, using $(f_{2})$, H\"{o}lder inequality and embedding theorem (proposition 2.3),
\begin{align*}
\begin{split}
c=&\underset{m\rightarrow\infty}{\text{lim}}(I(u_{m,n})-\frac{1}{p^{+}}\langle I^{\prime}(u_{m,n}),u_{m,n}\rangle)\\
=&\underset{m\rightarrow\infty}{\text{lim}}\left[\int_{\Omega}\frac{|\nabla u_{m,n}|^{p(x)}}{p(x)}dx-\int_{\Omega}\frac{|u_{m,n}|^{q(x)}}{q(x)}dx-\int_{\Omega}F(x,u_{m,n})dx-\int_{\Omega}\mu_{n}u_{m,n}dx\right.\\
&\left.-\frac{1}{p^{+}}\left\lbrace \int_{\Omega}|\nabla u_{m,n}|^{p(x)}dx-\int_{\Omega}|u_{m,n}|^{q(x)}dx-\int_{\Omega}f(x,u_{m,n})u_{m,n}dx-\int_{\Omega}\mu_{n}u_{m,n}dx\right\rbrace\right] \\
\geq& \underset{m\rightarrow\infty}{\text{lim}}\int_{\Omega}\left(\frac{1}{p^{+}}-\frac{1}{q(x)}\right)|u_{m,n}|^{q(x)}dx-\int_{\Omega}\mu_{n}u_{m,n}dx+\frac{1}{p^{+}}\int_{\Omega}\mu_{n}u_{m,n}dx\\
\geq& \underset{m\rightarrow\infty}{\text{lim}}\int_{\Omega}\left(\frac{1}{p^{+}}-\frac{1}{q(x)}\right)|u_{m,n}|^{q(x)}dx-A^{\prime}\|\mu_{n}\|_{p^{\prime}(x)}\|u_{m,n}\|
\end{split}
\end{align*}
where, $A^{\prime}=(1-\frac{1}{p^{+}}) > 0$. This implies
\begin{equation*}
c+A^{\prime}\|\mu_{n}\|_{p^{\prime}(x)}\|u_{m,n}\|\geq \underset{m\rightarrow\infty}{\text{lim}}\int_{\Omega}\left(\frac{1}{p^{+}}-\frac{1}{q(x)}\right)|u_{m,n}|^{q(x)}dx
\end{equation*}
We already proved that $(u_{m,n})$ is bounded so let $M > 0$ be an upper bound of $\left(\|u_{m,n}\|\right) $ for some fixed $n$.
Define $q_{A_{\delta}}^{-}:=\underset{A{\delta}}{\text{inf}}\;q(x)$ so
\begin{align*}
\begin{split}
c+A^{\prime}\|\mu_{n}\|_{p^{\prime}(x)}M&\geq  \underset{m\rightarrow\infty}{\text{lim}}\int_{\Omega}\left(\frac{1}{p^{+}}-\frac{1}{q(x)}\right)|u_{m,n}|^{q(x)}dx\\
&\geq  \underset{m\rightarrow\infty}{\text{lim}}\int_{A_{\delta}}\left(\frac{1}{p^{+}}-\frac{1}{q(x)}\right)|u_{m,n}|^{q(x)}dx\\
&\geq \underset{m\rightarrow\infty}{\text{lim}}\int_{A_{\delta}}\left(\frac{1}{p^{+}}-\frac{1}{q_{A_{\delta}}^{-}}\right)|u_{m,n}|^{q(x)}dx\\
&=\left(\frac{1}{p^{+}}-\frac{1}{q_{A_{\delta}}^{-}}\right)\left(\int_{A_{\delta}}|u_{n}|^{q(x)}dx+\sum_{i\in J}\nu_{i}\right) \\
&\geq \left(\frac{1}{p^{+}}-\frac{1}{q_{A_{\delta}}^{-}}\right)\nu_{i}\\
&\geq \left(\frac{1}{p^{+}}-\frac{1}{q_{A_{\delta}}^{-}}\right)S^{N}
\end{split}
\end{align*}
Let us denote $A^{\prime}\|\mu_{n}\|_{p^{\prime}(x)}M=M^{\prime}$. Since, $\delta > 0$ is arbitrary and $q(x)$ is continuous we can say 
$c+M^{\prime}\geq \left(\frac{1}{p^{+}}-\frac{1}{q_{A}^{-}}\right)S^{N}$. This further implies $c\geq \left(\frac{1}{p^{+}}-\frac{1}{q_{A}^{-}}\right)S^{N}-M^{\prime}$.\\
Hence, for $c < \left(\frac{1}{p^{+}}-\frac{1}{q_{A}^{-}}\right)S^{N}-M^{\prime}$, index set $J$ is empty.
\end{proof}
\noindent We have proved that
$u_{m,n}\rightarrow u_{n}$ in $L^{q(x)}(\Omega)$ and $\nabla u_{m,n}\rightarrow \nabla u_{n}$ in $L^{p(x)}(\Omega)$ which were obtained by concentration compactness principle.
Since $p(x) < q(x)$ we have embedding of $L^{q(x)}(\Omega)$ in $L^{p(x)}(\Omega)$. Thus we get
$u_{m,n}\rightarrow u_{n}$ in $L^{p(x)}(\Omega)$ and $\nabla u_{m,n}\rightarrow \nabla u_{n}$ in $L^{p(x)}(\Omega)$. Hence $u_{m,n}\rightarrow u_{n}$ in $W_{0}^{1,p(x)}(\Omega)$.
Therefore, the functional $I_{n}$ satisfies the Palais-Smale condition.
\end{proof}
\noindent Therefore, from Lemma (3.1) and (3.2) we conclude that there exist critical point $u_{n} \in W_{0}^{1,p(x)}(\Omega)$ corresponding to each $\mu_{n}$ for the sequence of problems $(P_{n})$.\\
Now, choose a test function $v=T_k(u_n)$, where $T_k$ is a truncation operator defined as \[T_k(t)=\begin{cases}
 t, & |t| < k \\
 k{\text{sign}}(t), & |t |\geq k.
 \end{cases}\]
 Clearly $T_k(u_{n})\in W_0^{1,p(x)}(\Omega)$. Now 
 \begin{align*}
 \{|\nabla u_n|> t\} & = \{|\nabla u_n|> t,|u_n|\leq k\} \cup \{|\nabla u_n| > t,|u_n| > k\}
 \\& \subset \{|\nabla u_n|> t,|u_n|\leq k\} \cup \{|u_n| > k\}\subset \Omega.
 \end{align*}
 Hence, by the subadditivity of Lebesgue measure, we have
 \begin{equation}\label{eq5}
 |\{|\nabla u_n|> t\}| \leq |\{|\nabla u_n|> t,|u_n|\leq k\}| + |\{|u_n| > k\}|.
 \end{equation}
Hence we have
 \begin{align*}
 \int_\Omega |\nabla T_k(u_n)|^{p(x)}dx & \leq \lambda\int_\Omega |u_n|^{q(x)-2}u_n T_k(u_n)dx+\int_{\Omega}f(x,u_n)T_k(u_n)dx+\int_{\Omega}\mu_nT_k(u_n)dx
 \\& \leq k|\Omega|^{1/q(x)}\|u_n\|_{q(x)}^{q(x)/q'(x)}+\epsilon\int _{(|u_n| > T)}|u_n|^{q(x)-1}T_k(u_n)dx+\int _{\Omega\times[-T,T]}f(x,u_n)T_k(u_n)dx
 \\& ~~~+\int_{\Omega}\mu_nT_k(u_n)dx
 \\&  \leq C_1(q(x),\Omega)k+C_2(\epsilon,\Omega)k+k\int_{\Omega}\mu_n\;dx
 \\& \leq Ck,
 \end{align*}
 where we have used the condition $(f_3)$ to bound the second integral and the $L^1$-bound of the sequence $(\mu_n)$ to bound the third integral.Thus, $\|\nabla T_{k}(u_{n})\|^{\gamma}_{p(x)}\leq Ck\;\forall\;k > 1,$ where  \[\gamma=\begin{cases}
 p^{+}, & \|\nabla T_{k}(u_{n})\|_{p(x)} < 1\\
 p^{-}, & \|\nabla T_{k}(u_{n})\|_{p(x)} > 1.
 \end{cases}\]
Define $A_{1}=\{x\in\Omega:|u_n(x)| > k\}$. On using the Poincar\'{e} and the generalized H\"{o}lder inequality, we get
\begin{align*}
\begin{split}
k|\left\lbrace|u_{n}|>k\right\rbrace|&=\int_{\left\lbrace |u_n| > k \right\rbrace} |T_k(u_n)|dx\\
&\leq \int_{\Omega}|T_{k}(u_{n})|dx\leq \left(\frac{1}{p^{-}}-\frac{1}{p^{-\prime}}\right)(|\Omega|+1)^{\frac{1}{p^{-\prime}}}\|T_{k}(u_{n})\|_{p(x)}\leq C_{3}k^{\frac{1}{\gamma}}
\end{split}
\end{align*} 
From this we get, $|\lbrace|u_{n}| > k\rbrace|\leq \frac{c_{3}}{k^{1-\frac{1}{\gamma}}}\;\forall\; k > 1$.
Hence, $(u_{n})$ is bounded in $M^{1-\frac{1}{\gamma}}(\Omega)$.\\
Now again on restricting the integral over the set defined as $A_{2}=\{x\in\Omega:|u_n(x)|\leq k\}$.\\
\begin{align*}
\begin{split}
\int_{\lbrace|u_{n}|\leq k\rbrace}|\nabla T_{k}(u_{n})|^{p(x)}dx&=\int_{\lbrace|u_{n}|\leq k\rbrace}|\nabla(u_{n})|^{p(x)}dx\\
&\geq \int_{\lbrace|\nabla u_{n}| > t, |u_{n}|\leq k\rbrace}|\nabla u_{n}|^{p(x)}dx\\
&\geq \int_{\lbrace|\nabla u_{n}| > t, |u_{n}|\leq k\rbrace}|t|^{p(x)}dx\\
&\geq t^{p^{-}}|\lbrace|\nabla u_{n}| > t, |u_{n}|\leq k\rbrace|
\end{split}
\end{align*}
Thus $Ck\geq t^{p^{-}}|\lbrace|\nabla u_{n}| > t, |u_{n}|\leq k\rbrace|$ which implies $|\lbrace|\nabla u_{n}| > t, |u_{n}|\leq k\rbrace| \leq \frac{Ck}{t^{p^{-}}}$.
Hence, from (3.8) we have $\lbrace|\nabla u_{n}| > t\rbrace\leq \frac{Ck}{t^{p^{-}}}+\frac{C_{3}}{k^{1-\frac{1}{\gamma}}}\;\forall\; k > 1$.\\
On choosing $k=t^{\frac{\gamma p^{-}}{2\gamma-1}}$ we obtain $|\lbrace|\nabla u_{n}| > t\rbrace|\leq \frac{C_{4}}{t^{\frac{(\gamma -1)p^{-}}{2\gamma-1}}}\;\forall\;t\geq 1$, where $C_{4}={\text{max}}\lbrace C , C_{3}\rbrace$.
This implies that $(\nabla u_{n})$ is bounded in $M^{{\frac{(\gamma -1)p^{-}}{2\gamma-1}}}(\Omega)$. Then $(u_{n})$ is bounded in $W_{0}^{1,s(x)}(\Omega)$ for $s(x) < s$, where $s={\text{min}}\left\lbrace 1-\frac{1}{\gamma},{\frac{(\gamma -1)p^{-}}{2\gamma-1}}\right\rbrace$. We know that $W_{0}^{1,s(x)}(\Omega)$ is again a reflexive space. Hence on repeating the arguments used to prove $u_{m,n}\rightarrow u_{n}$ in $W_{0}^{1,p(x)}(\Omega)$ we find that $u_{n}\rightarrow u$ in $W_{0}^{1,s(x)}(\Omega)$. This limit $u$ is a nontrivial  weak solution of problem $(P)$ in $W_{0}^{1,s(x)}(\Omega)$.

\section*{Acknowledgement}
The author  Amita Soni thanks the Department of
Science and Technology (D. S. T), Govt. of India for financial
support. Both the authors also acknowledge the facilities received
from the Department of Mathematics, National Institute of Technology
Rourkela.

 {\sc Amita Soni} and {\sc D. Choudhuri}\\
Department of Mathematics,\\
National Institute of Technology Rourkela, Rourkela - 769008,
India\\
e-mails: soniamita72@gmail.com and dc.iit12@gmail.com.

\end{document}